\documentclass{article}

\usepackage{amsmath,amssymb,amsthm}

\newtheorem{dfn}{Definition}
\newtheorem{thm}[dfn]{Theorem}

\newtheorem{cor}[dfn]{Corollary}
\newtheorem{lem}[dfn]{Lemma}
\newtheorem{remark}[dfn]{Remark}

\title{\bf Generalized weighted trapezoid and Gr\"uss type inequalities on time scales\footnote{This is a preprint of a paper whose final and definite form is published open access in the Australian Journal of Mathematical Analysis and Applications. Cite this paper as ``E. R. Nwaeze, Generalized weighted trapezoid and Gr\"uss type inequalities on time scales, Aust. J. Math. Anal. Appl., {\bf 11}(1)(2017), Article 4, 1–13."} }

\author{Eze R. Nwaeze\\
{\tt enwaeze@mytu.tuskegee.edu}}

\date{Department of Mathematics, Tuskegee University,\\
Tuskegee, AL 36088, USA}

\begin{document}
{\small
\maketitle

\begin{abstract}

In this work, we obtain some new generalized weighted trapezoid and Gr\"uss type inequalities on time scales for parameter functions. Our results give a broader generalization of the results due to Pachpatte in \cite{Pach}. In addition, the continuous and discrete cases are also considered from which, other results are obtained.

\bigskip


\medskip

\end{abstract}


\section{Introduction}
In 2003, Pachpatte \cite{Pach} obtained the following versions (see also \cite{Mit1,Mit2} for the original versions) of the trapezoid and Gr\"uss type inequalities:

\begin{thm}\label{Pach1}
Let $f:[a, b]\rightarrow \mathbb{R}$ be continuous on $[a, b]$ and differentiable on $(a, b),$ whose derivative $f':(a, b)\rightarrow \mathbb{R}$ is bounded on $(a, b).$ Then $$\Bigg|\frac{1}{2}\Big(f^2(b) - f^2(a)\Big) - \dfrac{f(b)-f(a)}{b-a}\int_a^b f(x)dx\Bigg|\leq \frac{1}{3}(b-a)^2||f'||^2_{\infty}.$$
\end{thm}

\begin{thm}\label{Pach2}
Let $f, g:[a, b]\rightarrow \mathbb{R}$ be continuous on $[a, b]$ and differentiable on $(a, b),$ whose derivative $f', g':(a, b)\rightarrow \mathbb{R}$ are bounded on $(a, b).$ Then
\begin{align*}
&\Bigg|\dfrac{1}{b-a}\int_a^b f(x)g(x)dx - \Bigg(\dfrac{1}{b-a}\int_a^b f(x)dx\Bigg)\Bigg(\dfrac{1}{b-a}\int_a^b g(x)dx\Bigg)\Bigg|\\
&\leq \dfrac{1}{2(b-a)^2}\int_a^b\Big(||f'||_{\infty}|g(x)|+||g'||_{\infty}|f(x)| \Big)E(x) dx,
\end{align*}
where 
$E(x)=\frac{1}{4}(b-a)^2+\Big(x-\frac{a+b}{2}\Big)^2,$ for $x\in [a, b].$
\end{thm}

In 1988, Hilger \cite{Hilger} introduced the concept of time scale as a unifier between the continuous and discrete calculus. Since then, many researchers have been able to extend known classical integral inequalities to time scales (see for example \cite{Karpuz,Liu,Tuna1,Tuna2,Tuna3}). In 2009, Ng\^o and Liu \cite{Liu2} gave a sharp Gr\"uss inequality on time scales. For the sake of this work, we present the following results of Liu and Tuna \cite{Tuna4} which are generalizations of the trapezoid and Gr\"uss type inequalities on time scales.

\begin{thm}\label{WT1}
Let $0\leq k\leq 1, g:[a, b]\rightarrow [0,\infty)$ be continuous and positive and $h:[a, b]\rightarrow \mathbb{R}$ be differentiable such that $h^{\Delta}(t )=g(t)$ on $[a, b].$ Suppose also that $a, b, s, t\in \mathbb{T},$ ~$a<b,$ and $f:[a, b]\rightarrow \mathbb{R}$ is differentiable. Then the following inequality holds
\begin{align}
\nonumber
 &\Bigg|(1-k)\Big(f^2(b)- f^2(a) \Big)+k\Big(f(\sigma(a))f(b) - f(\sigma(b))f(a)\Big) \\
\nonumber
&+ \Bigg[ (k-1)\frac{f(b)-f(a)}{\int_a^b g(t)\Delta t} + k\frac{f(\sigma(b)) - f(\sigma(a))}{\int_a^b g(t)\Delta t}\Bigg]\int_a^b g(t)f(\sigma(t))\Delta t\\
\nonumber
&- \frac{f(b)-f(a)}{\int_a^b g(t)\Delta t}\int_a^b g(t)f(\sigma^2(t))\Delta t\Bigg|\\
\nonumber
&\leq \dfrac{M(N+M)}{\int_a^b g(t)\Delta t} \int_a^b\Bigg(\int_a^b \big|S(t, s)\big|\Delta s\Bigg)\Delta t,
\end{align}
where
\begin{equation*}\label{WT2}
S(t, s)=
\begin{cases}
h(s)-\Big((1-k)h(a)+kh(t)\Big), ~~~~s\in[a, t),\\
h(s)-\Big(kh(t)+(1-k)h(b)\Big), ~~~~s\in[t, b],
\end{cases}
\end{equation*}
$M=\sup\limits_{a<t<b}\Big|f^{\Delta}(t)\Big|<\infty ~~and~~ N=\sup\limits_{a<t<b}\Big|f^{\Delta}(\sigma(t))\Big|<\infty.$
\end{thm}

\begin{thm}\label{WT3}
Let $0\leq k\leq 1, g:[a, b]\rightarrow [0,\infty)$ be continuous and positive and $h:[a, b]\rightarrow \mathbb{R}$ be differentiable such that $h^{\Delta}(t )=g(t)$ on $[a, b].$ Suppose also that $a, b, s, t\in \mathbb{T},$ ~$a<b,$ and $p, q:[a, b]\rightarrow \mathbb{R}$ are differentiable. Then the following inequality holds
\begin{align*}
\nonumber
&\Bigg| 2(1-k)\Bigg(\int_a^bg(t)\Delta t\Bigg)\Bigg(\int_a^bp(t)q(t)\Delta t\Bigg)\\
&+k\int_a^b\Bigg\{q(t)\Bigg[\Bigg(\int_a^t g(s)\Delta s\Bigg)p(a) + \Bigg(\int_t^b g(s)\Delta s\Bigg)p(b) \Bigg] \\
&+ p(t)\Bigg[\Bigg(\int_a^t g(s)\Delta s\Bigg)q(a) + \Bigg(\int_t^b g(s)\Delta s\Bigg)q(b) \Bigg] \Bigg\}\Delta t\\
&-\Bigg[\Bigg(\int_a^b q(t)\Delta t\Bigg)\Bigg(\int_a^b g(t)p(\sigma(t))\Delta t\Bigg)  + \Bigg(\int_a^b p(t)\Delta t\Bigg)\Bigg(\int_a^b g(t)q(\sigma(t))\Delta t\Bigg) \Bigg]\Bigg|\\
&\leq \int_a^b\Big(P|q(t)| + Q|p(t)|\Big)\Bigg(\int_a^b \big|S(t, s)\big|\Delta s\Bigg)\Delta t,
\end{align*}
where
\begin{equation*}\label{WT4}
S(t, s)=
\begin{cases}
h(s)-\Big((1-k)h(a)+kh(t)\Big), ~~~~s\in[a, t),\\
h(s)-\Big(kh(t)+(1-k)h(b)\Big), ~~~~s\in[t, b],
\end{cases}
\end{equation*}
$P=\sup\limits_{a<t<b}\Big|p^{\Delta}(t)\Big|<\infty ~~and~~ Q=\sup\limits_{a<t<b}\Big|q^{\Delta}(t)\Big|<\infty.$
\end{thm}
Recently, Xu and Fang \cite{XuFang} introduced a technique of parameter functions. In light of this, they obtained a new Ostrowski type inequality for parameter functions. Inspired by this technique and  the idea used in \cite{Tuna4}, we prove another version of the trapezoid and Gr\"uss type inequalities for parameter functions via a new weighted Peano Kernel.

The paper is arranged as follows. In Section~\ref{sec:Prelim}, we recall necessary results and definitions in time scale theory. Our results are formulated and proved in Section~\ref{sec:MR}.

\section{Preliminaries}\label{sec:Prelim}
We start by presenting the following time scale essentials that will come handy in what follows. For more on the theory of time scales, we refer the reader to the books of Bohner and Peterson \cite{Bohner1} and  Bohner and Peterson \cite{Bohner2}. 

\begin{dfn}

~~A {\it time scale} $\mathbb{T}$ is an arbitrary nonempty closed subset of $\mathbb{R}.$ The forward {\it jump operator} $\sigma:\mathbb{T}\rightarrow\mathbb{T}$ and backward {\it jump operator} $\rho:\mathbb{T}\rightarrow\mathbb{T}$ are defined by $\sigma(t):=\inf\{s\in\mathbb{T}: s>t\}$ for $t\in\mathbb{T}$ and $\rho(t):=\sup\{s\in\mathbb{T}: s<t\}$ for $t\in\mathbb{T},$ respectively.
 Clearly, we see that $\sigma(t)\geq t$ and $\rho(t)\leq t$ for all $t\in \mathbb{T}.$ If $\sigma(t)>t,$ then we say that $t$ is right-scattered, while if $\rho(t)<t,$ then we say that $t$ is left-scattered. If $\sigma(t)=t,$ then $t$ is called right dense, and if $\rho(t)=t$ then $t$ is called left dense. Points that are both right dense and left dense are called dense. The set ${\mathbb{T}}^k$ is defined as follows: if $\mathbb{T}$ has a left scattered maximum $m,$ then ${\mathbb{T}}^k=\mathbb{T}-m;$ otherwise, ${\mathbb{T}}^k=\mathbb{T}.$
For $a, b \in \mathbb{T}$ with $a\leq b,$ we define the interval $[a, b]$ in $\mathbb{T}$ by $[a, b]=\{t\in\mathbb{T}: a\leq t\leq b\}.$ Open intervals and half-open intervals are defined in the same manner.
\end{dfn}

\begin{dfn}
The function $f:\mathbb{T}\rightarrow \mathbb{R},$ is called differentiable at $t\in\mathbb{T}^k,$ with delta derivative $f^{\Delta}(t)\in \mathbb{R},$ if for any given $\epsilon>0$ there exist a neighborhood $U$ of $t$ such that $$\left|f(\sigma(t))-f(s)-f^{\Delta}(t)(\sigma(t)-s)\right|\leq \epsilon |\sigma(t)-s|,~~~~\forall s\in U.$$
\end{dfn}
If $\mathbb{T}=\mathbb{R},$ then $f^{\Delta}(t)=\dfrac{df(t)}{dt},$ and if $\mathbb{T}=\mathbb{Z},$ then $f^{\Delta}(t)=f(t+1)-f(t).$

\begin{thm}\label{pro}
Let $f, g:\mathbb{T}\rightarrow\mathbb{R}$ be two differentiable functions at $t\in {\mathbb{T}}^k.$ Then the product $fg:\mathbb{T}\rightarrow\mathbb{R}$ is also differentiable at $t$ with 
$$(fg)^{\Delta}(t)=f^{\Delta}(t)g(t)+f(\sigma(t))g^{\Delta}(t)=f(t)g^{\Delta}(t)+f^{\Delta}(t)g(\sigma(t)).$$
\end{thm}

\begin{dfn}
The function $f:\mathbb{T}\rightarrow \mathbb{R}$ is said to be $rd-$continuous if it is continuous at all dense points $t\in\mathbb{T}$ and its left-sided limits exist at all left dense points $t\in\mathbb{T}.$
\end{dfn}

\begin{dfn}\label{def2}
Let $f$ be a $rd-$continuous function. Then $g:\mathbb{T}\rightarrow\mathbb{R}$ is called the antiderivative of $f$ on $\mathbb{T}$ if it is differentiable on $\mathbb{T}$ and satisfies $g^{\Delta}(t)=f(t)$ for any $t\in {\mathbb{T}}^k.$ In this case, we have
$$\int_a^b f(s)\Delta s = g(b)-g(a).$$
\end{dfn}

\begin{thm}\label{A1}
If $a,b,c \in \mathbb{T}$ with $a<c<b,$ $\alpha\in\mathbb{R}$ and $f, g$ are $rd-$continuous, then
\begin{enumerate}
\item [\rm (i)] $\int_a^b [f(t)+g(t)]\Delta t=\int_a^b f(t)\Delta t + \int_a^b g(t)\Delta t.$
\item [\rm (ii)] $\int_a^b \alpha f(t)\Delta t= \alpha \int_a^b f(t)\Delta t$
\item [\rm (iii)] $\int_a^b f(t)\Delta t= -\int_b^a f(t)\Delta t$
\item [\rm (iv)] $\int_a^b f(t)\Delta t=\int_a^c f(t)\Delta t + \int_c^b f(t)\Delta t.$
\item [\rm (v)] $\left|\int_a^b f(t)\Delta t\right| \leq \int_a^b |f(t)|\Delta t$ for all $t\in [a, b].$
\item [\rm (vi)] $\int_a^b f(t)g^{\Delta}(t)\Delta t=(fg)(b) - (fg)(a) - \int_a^b f^{\Delta}(t)g(\sigma(t))\Delta t.$
\end{enumerate}
\end{thm}

\begin{dfn}
Let $h_k:\mathbb{T}^2\rightarrow\mathbb{T},$~$k\in \mathbb{N}$ be functions that are recursively defined as $$h_0(t, s)=1$$ and $$h_{k+1}(t,s)=\int_s^t h_k(\tau, s)\Delta \tau,~~{for~all}~~s, t \in \mathbb{T}.$$
\end{dfn}
When $\mathbb{T}=\mathbb{R},$ then for all $s, t \in \mathbb{T},$ $$h_k(t, s)=\frac{(t-s)^k}{k!}.$$


\section{Main results}\label{sec:MR}
For the proof of our main results, we will need the following lemma due to Nwaeze \cite{Nwaeze}.

\begin{lem}[A weighted generalized Montgomery Identity]\label{Wlem1}
Let $\nu:[a, b]\rightarrow [0,\infty)$ be $rd-$continuous and positive and $w:[a, b]\rightarrow \mathbb{R}$ be differentiable such that $w^{\Delta}(t )=\nu(t)$ on $[a, b].$ Suppose also that $a, b, s, t\in \mathbb{T},$ ~$a<b,$~$f:[a, b]\rightarrow \mathbb{R}$ is differentiable, and $\psi$ is a function of $[0, 1]$ into $[0, 1].$ Then we have the following equation
\begin{align}\label{Weq}
\nonumber
 \left[\dfrac{1+\psi(1-\lambda)-\psi(\lambda)}{2}f(t) + \dfrac{\psi(\lambda)f(a)+\left(1-\psi(1-\lambda)\right)f(b)}{2}\right]\int_a^b \nu(t)\Delta t\\
=\int_a^b K(s,t)f^{\Delta}(s)\Delta s + \int_a^b \nu(s)f(\sigma(s))\Delta s,
\end{align}
where
\begin{equation}\label{WE}
K(s, t)=
\begin{cases}
w(s)-\left(w(a)+\psi(\lambda)\frac{w(b)-w(a)}{2}\right), ~~~~s\in[a, t),\\
w(s)-\left(w(a)+(1+\psi(1-\lambda))\frac{w(b)-w(a)}{2}\right), ~~~~s\in[t, b].
\end{cases}
\end{equation}
\end{lem}

\subsection{A weighted trapezoid type inequality on time scales}

\begin{thm}\label{Trap}
Let $\nu:[a, b]\rightarrow [0,\infty)$ be continuous and positive and $w:[a, b]\rightarrow \mathbb{R}$ be differentiable such that $w^{\Delta}(t )=\nu(t)$ on $[a, b].$ Suppose also that $a, b, s, t\in \mathbb{T},$ ~$a<b,$~$f:[a, b]\rightarrow \mathbb{R}$ is differentiable, and $\psi$ is a function of $[0, 1]$ into $[0, 1].$ Then we have the following inequality

\begin{align}
\nonumber
 &\Bigg|\dfrac{1+\psi(1-\lambda)-\psi(\lambda)}{2}\Big(f^2(b)- f^2(a) \Big) -\dfrac{f(b) - f(a)}{\int_a^b \nu(t)\Delta t}\int_a^b \nu(s)\Big(f(\sigma(s))+ f(\sigma^2(s)) \Big)\Delta s\\
\nonumber
&+ \dfrac{\psi(\lambda)\Big(f(a)+ f(\sigma(a))\Big) +\left(1-\psi(1-\lambda)\right)\Big(f(b)+f(\sigma(b))\Big)}{2}\Big(f(b)- f(a)\Big)\Bigg|\\
&\leq \dfrac{M(N+M)}{\int_a^b \nu(t)\Delta t} \int_a^b\Bigg(\int_a^b \big|K(s,t)\big|\Delta s\Bigg)\Delta t,
\end{align}
where
\begin{equation}
K(s, t)=
\begin{cases}
w(s)-\left(w(a)+\psi(\lambda)\frac{w(b)-w(a)}{2}\right), ~~~~s\in[a, t),\\
w(s)-\left(w(a)+(1+\psi(1-\lambda))\frac{w(b)-w(a)}{2}\right), ~~~~s\in[t, b],
\end{cases}
\end{equation}
$M=\sup\limits_{a<t<b}\Big|f^{\Delta}(t)\Big|<\infty ~~and~~ N=\sup\limits_{a<t<b}\Big|f^{\Delta}(\sigma(t))\Big|<\infty.$
\end{thm}

\begin{proof}
From Lemma \ref{Wlem1}, we have
\begin{align}\label{eqn1}
\nonumber
 \dfrac{1+\psi(1-\lambda)-\psi(\lambda)}{2}f(t)& =\dfrac{1}{\int_a^b \nu(t)\Delta t}\int_a^b K(s,t)f^{\Delta}(s)\Delta s + \dfrac{1}{\int_a^b \nu(t)\Delta t}\int_a^b \nu(s)f(\sigma(s))\Delta s\\
& - \dfrac{\psi(\lambda)f(a)+\left(1-\psi(1-\lambda)\right)f(b)}{2},
\end{align}
and 
\begin{align}\label{eqn2}
\nonumber
 \dfrac{1+\psi(1-\lambda)-\psi(\lambda)}{2}f(\sigma(t))& =\dfrac{1}{\int_a^b \nu(t)\Delta t}\int_a^b K(s,t)f^{\Delta}(\sigma(s))\Delta s + \dfrac{1}{\int_a^b \nu(t)\Delta t}\int_a^b \nu(s)f(\sigma^2(s))\Delta s\\
& - \dfrac{\psi(\lambda)f(\sigma(a))+\left(1-\psi(1-\lambda)\right)f(\sigma(b))}{2}.
\end{align}
Adding Equations (\ref{eqn1}) and (\ref{eqn2}), we get
\begin{align}\label{eqn3}
\nonumber
 \dfrac{1+\psi(1-\lambda)-\psi(\lambda)}{2}\Big(f(t) &+ f(\sigma(t))\Big) = \dfrac{1}{\int_a^b \nu(t)\Delta t}\int_a^b K(s,t)\Big(f^{\Delta}(s) + f^{\Delta}(\sigma(s))\Big)\Delta s \\
\nonumber
&+ \dfrac{1}{\int_a^b \nu(t)\Delta t}\int_a^b \nu(s)\Big(f(\sigma(s))+ f(\sigma^2(s)) \Big)\Delta s\\
& - \dfrac{\psi(\lambda)\Big(f(a)+ f(\sigma(a))\Big) +\left(1-\psi(1-\lambda)\right)\Big(f(b)+f(\sigma(b))\Big)}{2}.
\end{align}
Multiplying (\ref{eqn3}) by $f^{\Delta}(t)$ and using Theorem \ref{pro} gives

\begin{align}\label{eqn4}
\nonumber
 \dfrac{1+\psi(1-\lambda)-\psi(\lambda)}{2}\big(f^2\big)^{\Delta}(t) &= \dfrac{1}{\int_a^b \nu(t)\Delta t} f^{\Delta}(t) \int_a^b K(s,t)\Big(f^{\Delta}(s) + f^{\Delta}(\sigma(s))\Big)\Delta s\\
\nonumber
&+ \dfrac{1}{\int_a^b \nu(t)\Delta t} f^{\Delta}(t) \int_a^b \nu(s)\Big(f(\sigma(s))+ f(\sigma^2(s)) \Big)\Delta s\\
& - \dfrac{\psi(\lambda)\Big(f(a)+ f(\sigma(a))\Big) +\left(1-\psi(1-\lambda)\right)\Big(f(b)+f(\sigma(b))\Big)}{2}f^{\Delta}(t).
\end{align}

Now, integrating (\ref{eqn4}) on $[a, b],$ we have

\begin{align}\label{eqn5}
\nonumber
 \dfrac{1+\psi(1-\lambda)-\psi(\lambda)}{2}&\Big(f^2(b)- f^2(a) \Big)= \dfrac{1}{\int_a^b \nu(t)\Delta t} \int_a^bf^{\Delta}(t) \Bigg[\int_a^b K(s,t)\Big(f^{\Delta}(s) + f^{\Delta}(\sigma(s))\Big)\Delta s\Bigg]\Delta t\\
\nonumber
&+ \dfrac{f(b) - f(a)}{\int_a^b \nu(t)\Delta t}\int_a^b \nu(s)\Big(f(\sigma(s))+ f(\sigma^2(s)) \Big)\Delta s\\
& - \dfrac{\psi(\lambda)\Big(f(a)+ f(\sigma(a))\Big) +\left(1-\psi(1-\lambda)\right)\Big(f(b)+f(\sigma(b))\Big)}{2}\Big(f(b)- f(a)\Big).
\end{align}

This implies

\begin{align}\label{eqn6}
\nonumber
 &\dfrac{1+\psi(1-\lambda)-\psi(\lambda)}{2}\Big(f^2(b)- f^2(a) \Big) -\dfrac{f(b) - f(a)}{\int_a^b \nu(t)\Delta t}\int_a^b \nu(s)\Big(f(\sigma(s))+ f(\sigma^2(s)) \Big)\Delta s\\
\nonumber
&+ \dfrac{\psi(\lambda)\Big(f(a)+ f(\sigma(a))\Big) +\left(1-\psi(1-\lambda)\right)\Big(f(b)+f(\sigma(b))\Big)}{2}\Big(f(b)- f(a)\Big)\\
&= \dfrac{1}{\int_a^b \nu(t)\Delta t} \int_a^bf^{\Delta}(t) \Bigg[\int_a^b K(s,t)\Big(f^{\Delta}(s) + f^{\Delta}(\sigma(s))\Big)\Delta s\Bigg]\Delta t.
\end{align}

Taking the absolute value of both sides of (\ref{eqn6}) and using item (v) of Theorem \ref{A1}, we get the desired result.
\end{proof}

\begin{cor}\label{cor1}
For $\mathbb{T}=\mathbb{R}$ in Theorem \ref{Trap} we get

\begin{align}
\nonumber
 &\Bigg|\dfrac{1+\psi(1-\lambda)-\psi(\lambda)}{4}\Big(f^2(b)- f^2(a) \Big) -\dfrac{f(b) - f(a)}{\int_a^b \nu(t) dt}\int_a^b \nu(s)f(s) ds\\
\nonumber
&+ \dfrac{\psi(\lambda)f(a) +\left(1-\psi(1-\lambda)\right)f(b)}{2}\Big(f(b)- f(a)\Big)\Bigg|\\
&\leq \dfrac{M^2}{\int_a^b \nu(t) dt} \int_a^b\Bigg(\int_a^b \big|K(s,t)\big| ds\Bigg) dt,
\end{align}
where $\nu(t)=w'(t)$ on $[a, b],$
\begin{equation}
K(s, t)=
\begin{cases}
w(s)-\left(w(a)+\psi(\lambda)\frac{w(b)-w(a)}{2}\right), ~~~~s\in[a, t),\\
w(s)-\left(w(a)+(1+\psi(1-\lambda))\frac{w(b)-w(a)}{2}\right), ~~~~s\in[t, b],
\end{cases}
\end{equation}
and 
$M=\sup\limits_{a<t<b}\Big|f'(t)\Big|<\infty.$ 

\end{cor}

\begin{cor}\label{cor2}
For $w(t)=t,$ we have that $\nu(t)=1.$ Using this, the inequality in Theorem \ref{Trap} becomes

\begin{align}
\nonumber
 &\Bigg|\dfrac{1+\psi(1-\lambda)-\psi(\lambda)}{2}\Big(f^2(b)- f^2(a) \Big) -\dfrac{f(b) - f(a)}{b-a}\int_a^b \Big(f(\sigma(s))+ f(\sigma^2(s)) \Big)\Delta s\\
\nonumber
&+ \dfrac{\psi(\lambda)\Big(f(a)+ f(\sigma(a))\Big) +\left(1-\psi(1-\lambda)\right)\Big(f(b)+f(\sigma(b))\Big)}{2}\Big(f(b)- f(a)\Big)\Bigg|\\
\nonumber
&\leq \dfrac{M(N+M)}{b-a} \int_a^b\Bigg[h_2\left(a, a+\psi(\lambda)\frac{b-a}{2} \right) + h_2\left(t, a+\psi(\lambda)\frac{b-a}{2} \right) \\
&+ h_2\left(t, a+(1+\psi(1-\lambda))\frac{b-a}{2} \right)+  h_2\left(b, a+(1+\psi(1-\lambda))\frac{b-a}{2} \right)\Bigg]\Delta t,
\end{align}
for all $\lambda\in [0, 1]$ such that ~$a+\psi(\lambda)\frac{b-a}{2}$ ~and ~$a+(1+\psi(1-\lambda))\frac{b-a}{2}$ are in $\mathbb{T},$ and \\$t\in \left[a+\psi(\lambda)\frac{b-a}{2}, a+(1+\psi(1-\lambda))\frac{b-a}{2}\right].$
Here,
\begin{equation}
K(s, t)=
\begin{cases}
s-\left(a+\psi(\lambda)\frac{b-a}{2}\right), ~~~~s\in[a, t),\\
s-\left(a+(1+\psi(1-\lambda))\frac{b-a}{2}\right), ~~~~s\in[t, b],
\end{cases}
\end{equation}
$M=\sup\limits_{a<t<b}\Big|f^{\Delta}(t)\Big|<\infty ~~and~~ N=\sup\limits_{a<t<b}\Big|f^{\Delta}(\sigma(t))\Big|<\infty.$

\end{cor}
\begin{proof}
Here, we only need to justify the right hand side of the inequality. We proceed as follows.
\begin{align*}
\int_a^b |K(s, t)|\Delta s &= \int_a^t |K(s, t)|\Delta s +\int_t^b |K(s, t)|\Delta s\\
&= \int_a^t \left|s-\left(a+\psi(\lambda)\frac{b-a}{2}\right)\right|\Delta s + \int_t^b \left|s-\left(a+(1+\psi(1-\lambda))\frac{b-a}{2}\right)\right|\Delta s\\
&=\int_a^{a+\psi(\lambda)\frac{b-a}{2}} \left|s-\left(a+\psi(\lambda)\frac{b-a}{2}\right)\right|\Delta s + \int_{a+\psi(\lambda)\frac{b-a}{2}}^t \left|s-\left(a+\psi(\lambda)\frac{b-a}{2}\right)\right|\Delta s\\
&+\int_t^{a+(1+\psi(1-\lambda))\frac{b-a}{2}} \left|s-\left(a+(1+\psi(1-\lambda))\frac{b-a}{2}\right)\right|\Delta s\\
&+ \int_{a+(1+\psi(1-\lambda))\frac{b-a}{2}}^b \left|s-\left(a+(1+\psi(1-\lambda))\frac{b-a}{2}\right)\right|\Delta s\\
&=\int_{a+\psi(\lambda)\frac{b-a}{2}}^a \left[s-\left(a+\psi(\lambda)\frac{b-a}{2}\right)\right]\Delta s +  \int_{a+\psi(\lambda)\frac{b-a}{2}}^t \left[s-\left(a+\psi(\lambda)\frac{b-a}{2}\right)\right]\Delta s\\
&+\int_{a+(1+\psi(1-\lambda))\frac{b-a}{2}}^t \left[s-\left(a+(1+\psi(1-\lambda))\frac{b-a}{2}\right)\right]\Delta s\\
&+ \int_{a+(1+\psi(1-\lambda))\frac{b-a}{2}}^b \left[s-\left(a+(1+\psi(1-\lambda))\frac{b-a}{2}\right)\right]\Delta s\\
&=h_2\left(a, a+\psi(\lambda)\frac{b-a}{2} \right) + h_2\left(t, a+\psi(\lambda)\frac{b-a}{2} \right) + h_2\left(t, a+(1+\psi(1-\lambda))\frac{b-a}{2} \right)\\
&+  h_2\left(b, a+(1+\psi(1-\lambda))\frac{b-a}{2} \right).
\end{align*}
Hence, the result follows.
\end{proof}

\begin{cor}\label{cor3}
Taking $\psi(\lambda)=\lambda,$ in Corollary \ref{cor2} above yields

\begin{align}
\nonumber
 &\Bigg|(1-\lambda)\Big(f^2(b)- f^2(a) \Big) -\dfrac{f(b) - f(a)}{b-a}\int_a^b \Big(f(\sigma(s))+ f(\sigma^2(s)) \Big)\Delta s\\
\nonumber
&+ \dfrac{\lambda\Big(f(a)+ f(b) + f(\sigma(a)) + f(\sigma(b))\Big)}{2}\Big(f(b)- f(a)\Big)\Bigg|\\
\nonumber
&\leq \dfrac{M(N+M)}{b-a} \int_a^b\Bigg[h_2\left(a, a+\lambda\frac{b-a}{2} \right) + h_2\left(t, a+\lambda\frac{b-a}{2} \right) \\
&+ h_2\left(t, a+(2 - \lambda)\frac{b-a}{2} \right)+  h_2\left(b, a+(2 - \lambda)\frac{b-a}{2} \right)\Bigg]\Delta t,
\end{align}
for all $\lambda\in [0, 1]$ such that ~$a+\lambda\frac{b-a}{2}$ ~and ~$a+(2-\lambda)\frac{b-a}{2}$ are in $\mathbb{T},$ and \\$t\in \left[a+\lambda\frac{b-a}{2}, a+(2-\lambda)\frac{b-a}{2}\right].$
Here,
\begin{equation}
K(s, t)=
\begin{cases}
s-\left(a+\lambda\frac{b-a}{2}\right), ~~~~s\in[a, t),\\
s-\left(a+(2 - \lambda)\frac{b-a}{2}\right), ~~~~s\in[t, b],
\end{cases}
\end{equation}
$M=\sup\limits_{a<t<b}\Big|f^{\Delta}(t)\Big|<\infty ~~and~~ N=\sup\limits_{a<t<b}\Big|f^{\Delta}(\sigma(t))\Big|<\infty.$

\end{cor}

\begin{remark}
If we take $\lambda=0$ and $\mathbb{T}=\mathbb{R},$ Corollary \ref{cor3} reduces to Theorem \ref{Pach1}.
\end{remark}

\begin{cor}\label{cor4}
For the case when $\mathbb{T}=\mathbb{Z},$  Theorem \ref{Trap} becomes

\begin{align}
\nonumber
 &\Bigg|\dfrac{1+\psi(1-\lambda)-\psi(\lambda)}{2}\Big(f^2(b)- f^2(a) \Big) -\dfrac{f(b) - f(a)}{\sum\limits_{s=a}^{b-1} \nu(s)}\sum_{s=a}^{b-1} \nu(s)\Big(f(s+1)+ f(s+2) \Big)\\
\nonumber
&+ \dfrac{\psi(\lambda)\Big(f(a)+ f(a+1)\Big) +\left(1-\psi(1-\lambda)\right)\Big(f(b)+f(b+1)\Big)}{2}\Big(f(b)- f(a)\Big)\Bigg|\\
&\leq \dfrac{M(N+M)}{\sum\limits_{s=a}^{b-1} \nu(s)} \sum\limits_{t=a}^{b-1}\Bigg(\sum\limits_{s=a}^{b-1} \big|K(s,t)\big|\Bigg),
\end{align}
where $\nu(t)=w(t+1)-w(t)$ on $[a, b],$
\begin{equation}
K(s, t)=
\begin{cases}
w(s)-\left(w(a)+\psi(\lambda)\frac{w(b)-w(a)}{2}\right), ~~~~s\in[a, t),\\
w(s)-\left(w(a)+(1+\psi(1-\lambda))\frac{w(b)-w(a)}{2}\right), ~~~~s\in[t, b],
\end{cases}
\end{equation}
$M=\sup\limits_{a<t<b-1}\Big|\Delta f(t)\Big|<\infty ~~and~~ N=\sup\limits_{a<t<b-1}\Big|\Delta f(t+1)\Big|<\infty.$
\end{cor}

\subsection{A weighted Gr\"uss type inequality on time scales}

\begin{thm}\label{Gruss}
Let $\nu:[a, b]\rightarrow [0,\infty)$ be continuous and positive and $w:[a, b]\rightarrow \mathbb{R}$ be differentiable such that $w^{\Delta}(t )=\nu(t)$ on $[a, b].$ Suppose also that $a, b, s, t\in \mathbb{T},$ ~$a<b,$~$p, q:[a, b]\rightarrow \mathbb{R}$ is differentiable, and $\psi$ is a function of $[0, 1]$ into $[0, 1].$ Then we have the following inequality

\begin{align}
\nonumber
\Bigg| \Big(1+&\psi(1-\lambda)-\psi(\lambda)\Big)\Bigg(\int_a^b\nu(t)\Delta t\Bigg)\Bigg(\int_a^bp(t)q(t)\Delta t\Bigg)\\
\nonumber
&+\dfrac{\psi(\lambda)p(a)+\left(1-\psi(1-\lambda)\right)p(b)}{2}\Bigg(\int_a^b\nu(t)\Delta t\Bigg)\Bigg(\int_a^bq(t)\Delta t\Bigg)\\
\nonumber
&+\dfrac{\psi(\lambda)q(a)+\left(1-\psi(1-\lambda)\right)q(b)}{2}\Bigg(\int_a^b\nu(t)\Delta t\Bigg)\Bigg(\int_a^bp(t)\Delta t\Bigg)\\
\nonumber
 &- \int_a^bq(t)\Bigg(\int_a^b \nu(s)p(\sigma(s))\Delta s\Bigg)\Delta t - \int_a^bp(t)\Bigg(\int_a^b \nu(s)q(\sigma(s))\Delta s\Bigg)\Delta t\Bigg|\\
&\leq \int_a^b\Big(P|q(t)| + Q|p(t)|\Big)\Bigg(\int_a^b \big|K(s,t)\big|\Delta s\Bigg)\Delta t,
\end{align}
where
\begin{equation}
K(s, t)=
\begin{cases}
w(s)-\left(w(a)+\psi(\lambda)\frac{w(b)-w(a)}{2}\right), ~~~~s\in[a, t),\\
w(s)-\left(w(a)+(1+\psi(1-\lambda))\frac{w(b)-w(a)}{2}\right), ~~~~s\in[t, b],
\end{cases}
\end{equation}
$P=\sup\limits_{a<t<b}\Big|p^{\Delta}(t)\Big|<\infty ~~and~~ Q=\sup\limits_{a<t<b}\Big|q^{\Delta}(t)\Big|<\infty.$

\end{thm}
\begin{proof}
Applying Lemma \ref{Wlem1} to the differentiable functions $p$ and $q$, we obtain

\begin{align}\label{eqn11}
\nonumber
 \dfrac{1+\psi(1-\lambda)-\psi(\lambda)}{2}p(t)& =\dfrac{1}{\int_a^b \nu(t)\Delta t}\int_a^b K(s,t)p^{\Delta}(s)\Delta s + \dfrac{1}{\int_a^b \nu(t)\Delta t}\int_a^b \nu(s)p(\sigma(s))\Delta s\\
& - \dfrac{\psi(\lambda)p(a)+\left(1-\psi(1-\lambda)\right)p(b)}{2},
\end{align}
and 
\begin{align}\label{eqn12}
\nonumber
 \dfrac{1+\psi(1-\lambda)-\psi(\lambda)}{2}q(t)& =\dfrac{1}{\int_a^b \nu(t)\Delta t}\int_a^b K(s,t)q^{\Delta}(s)\Delta s + \dfrac{1}{\int_a^b \nu(t)\Delta t}\int_a^b \nu(s)q(\sigma(s))\Delta s\\
& - \dfrac{\psi(\lambda)q(a)+\left(1-\psi(1-\lambda)\right)q(b)}{2}.
\end{align}
Multiplying (\ref{eqn11}) by $q(t)$ and (\ref{eqn12}) by $p(t)$ and then adding the resulting identity gives
\begin{align}\label{eqn13}
\nonumber
 \Big(1+\psi(1-\lambda)&-\psi(\lambda)\Big)p(t)q(t)=\dfrac{1}{\int_a^b\nu(t)\Delta t}\Bigg[q(t)\int_a^b K(s,t)p^{\Delta}(s)\Delta s +p(t)\int_a^b K(s,t)q^{\Delta}(s)\Delta s \Bigg]\\
\nonumber
 &+ \dfrac{1}{\int_a^b \nu(t)\Delta t}\Bigg[q(t)\int_a^b \nu(s)p(\sigma(s))\Delta s + p(t)\int_a^b \nu(s)q(\sigma(s))\Delta s\Bigg] \\
 &- \dfrac{\psi(\lambda)p(a)+\left(1-\psi(1-\lambda)\right)p(b)}{2}q(t) -  \dfrac{\psi(\lambda)q(a)+\left(1-\psi(1-\lambda)\right)q(b)}{2}p(t).
\end{align}
Now integrating (\ref{eqn13}) on $[a, b]$ amounts to

\begin{align}\label{eqn14}
\nonumber
 \Big(1+&\psi(1-\lambda)-\psi(\lambda)\Big)\int_a^bp(t)q(t)\Delta t\\
\nonumber
&=\dfrac{1}{\int_a^b\nu(t)\Delta t}\Bigg[\int_a^bq(t)\Bigg(\int_a^b K(s,t)p^{\Delta}(s)\Delta s\Bigg)\Delta t+\int_a^bp(t)\Bigg(\int_a^b K(s,t)q^{\Delta}(s)\Delta s\Bigg)\Delta t \Bigg]\\
\nonumber
 &+ \dfrac{1}{\int_a^b \nu(t)\Delta t}\Bigg[\int_a^bq(t)\Bigg(\int_a^b \nu(s)p(\sigma(s))\Delta s\Bigg)\Delta t + \int_a^bp(t)\Bigg(\int_a^b \nu(s)q(\sigma(s))\Delta s\Bigg)\Delta t\Bigg] \\
 &- \dfrac{\psi(\lambda)p(a)+\left(1-\psi(1-\lambda)\right)p(b)}{2}\int_a^bq(t)\Delta t -  \dfrac{\psi(\lambda)q(a)+\left(1-\psi(1-\lambda)\right)q(b)}{2}\int_a^bp(t)\Delta t.
\end{align}
This implies that

\begin{align}\label{eqn15}
\nonumber
 \Big(1+&\psi(1-\lambda)-\psi(\lambda)\Big)\Bigg(\int_a^b\nu(t)\Delta t\Bigg)\Bigg(\int_a^bp(t)q(t)\Delta t\Bigg)\\
\nonumber
&=\int_a^bq(t)\Bigg(\int_a^b K(s,t)p^{\Delta}(s)\Delta s\Bigg)\Delta t+\int_a^bp(t)\Bigg(\int_a^b K(s,t)q^{\Delta}(s)\Delta s\Bigg)\Delta t \\
\nonumber
 &+ \int_a^bq(t)\Bigg(\int_a^b \nu(s)p(\sigma(s))\Delta s\Bigg)\Delta t + \int_a^bp(t)\Bigg(\int_a^b \nu(s)q(\sigma(s))\Delta s\Bigg)\Delta t\\
\nonumber
 &- \dfrac{\psi(\lambda)p(a)+\left(1-\psi(1-\lambda)\right)p(b)}{2}\Bigg(\int_a^b\nu(t)\Delta t\Bigg)\Bigg(\int_a^bq(t)\Delta t\Bigg)\\
&- \dfrac{\psi(\lambda)q(a)+\left(1-\psi(1-\lambda)\right)q(b)}{2}\Bigg(\int_a^b\nu(t)\Delta t\Bigg)\Bigg(\int_a^bp(t)\Delta t\Bigg).
\end{align}
Rearranging, taking absolute value and using item (v) of Theorem \ref{A1} yields

\begin{align}\label{eqn16}
\nonumber
\Bigg| \Big(1+&\psi(1-\lambda)-\psi(\lambda)\Big)\Bigg(\int_a^b\nu(t)\Delta t\Bigg)\Bigg(\int_a^bp(t)q(t)\Delta t\Bigg)\\
\nonumber
&+\dfrac{\psi(\lambda)p(a)+\left(1-\psi(1-\lambda)\right)p(b)}{2}\Bigg(\int_a^b\nu(t)\Delta t\Bigg)\Bigg(\int_a^bq(t)\Delta t\Bigg)\\
\nonumber
&+\dfrac{\psi(\lambda)q(a)+\left(1-\psi(1-\lambda)\right)q(b)}{2}\Bigg(\int_a^b\nu(t)\Delta t\Bigg)\Bigg(\int_a^bp(t)\Delta t\Bigg)\\
\nonumber
 &- \int_a^bq(t)\Bigg(\int_a^b \nu(s)p(\sigma(s))\Delta s\Bigg)\Delta t - \int_a^bp(t)\Bigg(\int_a^b \nu(s)q(\sigma(s))\Delta s\Bigg)\Delta t\Bigg|\\
&\leq \int_a^b\Bigg|q(t)\Bigg(\int_a^b K(s,t)p^{\Delta}(s)\Delta s\Bigg)+p(t)\Bigg(\int_a^b K(s,t)q^{\Delta}(s)\Delta s\Bigg)\Bigg|\Delta t.
\end{align}
Hence, the result follows.
\end{proof}

\begin{cor}\label{cor17}
For the case when $\mathbb{T}=\mathbb{R},$  Theorem \ref{Gruss} becomes
\begin{align}
\nonumber
\Bigg| \Big(1+&\psi(1-\lambda)-\psi(\lambda)\Big)\Bigg(\int_a^b\nu(t) dt\Bigg)\Bigg(\int_a^bp(t)q(t) dt\Bigg)\\
\nonumber
&+\dfrac{\psi(\lambda)p(a)+\left(1-\psi(1-\lambda)\right)p(b)}{2}\Bigg(\int_a^b\nu(t) dt\Bigg)\Bigg(\int_a^bq(t) dt\Bigg)\\
\nonumber
&+\dfrac{\psi(\lambda)q(a)+\left(1-\psi(1-\lambda)\right)q(b)}{2}\Bigg(\int_a^b\nu(t) dt\Bigg)\Bigg(\int_a^bp(t) dt\Bigg)\\
\nonumber
 &- \int_a^bq(t)\Bigg(\int_a^b \nu(s)p(s) ds\Bigg) dt - \int_a^bp(t)\Bigg(\int_a^b \nu(s)q(s) ds\Bigg) dt\Bigg|\\
&\leq \int_a^b\Big(P|q(t)| + Q|p(t)|\Big)\Bigg(\int_a^b \big|K(s,t)\big| ds\Bigg) dt,
\end{align}
where $\nu(t)=w'(t)$ on $[a, b],$
\begin{equation}
K(s, t)=
\begin{cases}
w(s)-\left(w(a)+\psi(\lambda)\frac{w(b)-w(a)}{2}\right), ~~~~s\in[a, t),\\
w(s)-\left(w(a)+(1+\psi(1-\lambda))\frac{w(b)-w(a)}{2}\right), ~~~~s\in[t, b],
\end{cases}
\end{equation}
$P=\sup\limits_{a<t<b}\Big|p'(t)\Big|<\infty ~~and~~ Q=\sup\limits_{a<t<b}\Big|q'(t)\Big|<\infty.$
\end{cor}

\begin{cor}\label{cor18}
For the case when $\mathbb{T}=\mathbb{Z},$  Theorem \ref{Gruss} amounts to
\begin{align}
\nonumber
\Bigg| \Big(1+&\psi(1-\lambda)-\psi(\lambda)\Big)\Bigg(\sum\limits_{t=a}^{b-1}\nu(t)\Bigg)\Bigg(\sum\limits_{t=a}^{b-1}p(t)q(t)\Bigg)\\
\nonumber
&+\dfrac{\psi(\lambda)p(a)+\left(1-\psi(1-\lambda)\right)p(b)}{2}\Bigg(\sum\limits_{t=a}^{b-1}\nu(t)\Bigg)\Bigg(\sum\limits_{t=a}^{b-1}q(t)\Bigg)\\
\nonumber
&+\dfrac{\psi(\lambda)q(a)+\left(1-\psi(1-\lambda)\right)q(b)}{2}\Bigg(\sum\limits_{t=a}^{b-1}\nu(t)\Bigg)\Bigg(\sum\limits_{t=a}^{b-1}p(t)\Bigg)\\
\nonumber
 &- \sum\limits_{t=a}^{b-1}q(t)\Bigg(\sum\limits_{s=a}^{b-1} \nu(s)p(s+1)\Bigg) - \sum\limits_{t=a}^{b-1}p(t)\Bigg(\sum\limits_{s=a}^{b-1} \nu(s)q(s+1)\Bigg)\Bigg|\\
&\leq \sum\limits_{t=a}^{b-1}\Big(P|q(t)| + Q|p(t)|\Big)\Bigg(\sum\limits_{s=a}^{b-1} \big|K(s,t)\big|\Bigg),
\end{align}
where $\nu(t)=w(t+1)-w(t)$ on $[a, b],$
\begin{equation}
K(s, t)=
\begin{cases}
w(s)-\left(w(a)+\psi(\lambda)\frac{w(b)-w(a)}{2}\right), ~~~~s\in[a, t),\\
w(s)-\left(w(a)+(1+\psi(1-\lambda))\frac{w(b)-w(a)}{2}\right), ~~~~s\in[t, b],
\end{cases}
\end{equation}
$P=\sup\limits_{a<t<b-1}\Big|\Delta p(t)\Big|<\infty ~~and~~ Q=\sup\limits_{a<t<b-1}\Big|\Delta q(t)\Big|<\infty.$
\end{cor}

\begin{cor}\label{cor19}
For the case when $\psi(\lambda)=\lambda,$ $w(t)=t$ and $\mathbb{T}=\mathbb{R},$  Theorem \ref{Gruss} boils down to 
\begin{align}
\nonumber
&\Bigg| 2(1-\lambda)(b-a)\Bigg(\int_a^bp(t)q(t) dt\Bigg)+\dfrac{\lambda(b-a)\Big(p(a)+p(b)\Big)}{2}\Bigg(\int_a^bq(t) dt\Bigg)\\
\nonumber
&+\dfrac{\lambda(b-a)\Big(q(a)+q(b)\Big)}{2}\Bigg(\int_a^bp(t) dt\Bigg)- 2\Bigg(\int_a^bq(t) dt\Bigg)\Bigg(\int_a^b p(t) dt\Bigg)\Bigg|\\
&\leq \int_a^b\Big(P|q(t)| + Q|p(t)|\Big)\Bigg(t^2-t(a+b)+\frac{a^2+b^2}{2}\Bigg) dt.
\end{align}
Here,
\begin{equation}
K(s, t)=
\begin{cases}
s-\left(a+\lambda\frac{b-a}{2}\right), ~~~~s\in[a, t),\\
s-\left(a+(2-\lambda)\frac{b-a}{2}\right), ~~~~s\in[t, b],
\end{cases}
\end{equation}
$P=\sup\limits_{a<t<b}\Big|p'(t)\Big|<\infty ~~and~~ Q=\sup\limits_{a<t<b}\Big|q'(t)\Big|<\infty.$
\end{cor}

\begin{remark}
If we take $\lambda=0,$ Corollary \ref{cor19} reduces to Theorem \ref{Pach2}.
\end{remark}




\end{document}